\documentclass[a4paper,11pt]{article}%
\usepackage{graphicx}
\usepackage{amsmath,amsxtra,amssymb,latexsym, amscd,amsthm}
\usepackage[mathscr]{eucal}
\newcommand{\R}{\mathbb{R}}
\newcommand{\Xt}{X_{\varepsilon,t}}
\newcommand{\Xs}{X_{\varepsilon,s}}

\newtheorem{thm}{Theorem}[section]

\newtheorem{lem}{Lemma}[section]
\newtheorem{prop}{Proposition}[section]
\theoremstyle{definition}

\newtheorem{assum}{Assumption}[section]
\theoremstyle{remark}
\newtheorem{rem}{Remark}[section]
\numberwithin{equation}{section}

\begin{document}
\title{Gaussian fluctuations for stochastic Volterra equations with small noise}
\author{Nguyen Tien Dung\thanks{Department of Mathematics, VNU University of Science, Vietnam National University, Hanoi, 334 Nguyen
Trai, Thanh Xuan, Hanoi, 084 Vietnam.}\,\,\footnote{Corresponding author. Email: dung\_nguyentien10@yahoo.com, dung@hus.edu.vn}
\and Nguyen Thu Hang\thanks{Department of Mathematics, Hanoi University of Mining and Geology, 18 Pho Vien, Bac Tu Liem, Hanoi, 084 Vietnam.}}

\date{\today}

\maketitle
\begin{abstract} In this paper,  we consider a general class of stochastic Volterra equations with small noise. Our aim is to study the fluctuation of the solution around its deterministic limit. We use the techniques of Malliavin calculus to show that the fluctuation process satisfies central limit theorem and provide an optimal estimate for the rate of convergence. An application to stochastic Volterra equations with fractional Brownian motion kernel is given to illustrate the theory.
\end{abstract}

\section{Introduction}
For every $\varepsilon\in (0,1)$, we consider stochastic Volterra equations of the form
\begin{equation}\label{hfk}
X_{\varepsilon,t}=x_0+\int_0^tb(t,s,X_{\varepsilon,s})ds+\varepsilon\int_0^t\sigma(t,s,X_{\varepsilon,s})dB_s, \mbox{  }0\le t\le T,
\end{equation}
where the initial data $x_0\in \R,\mbox{  }(B_t)_{t\in [0,T]}$ is a standard Brownian motion and $b,\sigma$ are measurable functions on $[0,T]^2\times \mathbb{R}.$ It is well known that the class of stochastic Volterra equations was first studied by Berger and Mizel in \cite{Berger}. Since then this class has been widely used in the modelling of numerous random phenomena such as fluid turbulence \cite{Chevillard2017}, DNA patterns \cite{Reynaud2010} and mathematical finance \cite{Abi2019b,Omar2019}, etc.

For small values $\varepsilon,$ the Volterra equation (\ref{hfk}) forms a dynamical system with small noise. We recall that the study of asymptotic behaviors of dynamical systems with small noise has a long history beginning in 1970 with the results of  Ventzell \& Freidlin \cite{Ventzell1970}. In this context, one of the fundamental problems is to study the convergence of the considering system to its deterministic limit. The convergence can be described via large deviation principle, central limit theorem and moderate deviation principle, etc.  In fact, the system (\ref{hfk}), the asymptotic behavior of  $X_{\varepsilon,t}$ as $\varepsilon\to 0$ has been investigated by various authors. However, most of papers are devoted to the large and moderate deviation results, see e.g.  \cite{Jacquier2022,Li2017,Nualart2000,Zhang2008} and references therein.

In this paper, our aim is to analyze the Gaussian fluctuation of $X_{\varepsilon,t}$ around its limit. More specifically, we prove a central limit theorem result with explicit estimates for the rate of convergence. It should be noted that, in the last years,  the Gaussian fluctuation for dynamical systems with small noise has been studied by several authors, see e.g. \cite{Bourguin2025,Dung2024,Fan2024} and references therein. We observe that, as $\varepsilon \mbox{ tends to } 0,$ the solution $(X_{\varepsilon,t})_{t\in[0,T]}$ of the equation (\ref{hfk}) converges to a deterministic process $(x_t)_{t\in[0,T]}$ solving the following integral equation
\begin{equation}\label{hfk1}
x_t=x_0+\int_0^tb(t,s,x_s)ds,\mbox{  }0\le t\le T.
\end{equation}
 Let us define the fluctuation process
\begin{equation}\label{hfk3}
\tilde{X}_{\varepsilon,t}:=\frac{X_{\varepsilon,t}-x_t}{\varepsilon},\mbox{  }0\le t\le T.
\end{equation}
It is known from \cite{Qiao2023} that $\tilde{X}_{\varepsilon,t}$ converges to a Gaussian limit $Y_t$ given by
\begin{equation}\label{hfk2}
Y_t=\int_0^tb'(t,s,x_s)Y_sds+\int_0^t\sigma(t,s,x_s)dB_s,\mbox{  }0\le t\le T.
\end{equation}
The results of \cite{Qiao2023} are qualitative and hence, the next important problem is to provide quantitative estimates for the rate of convergence via certain distances. Among others, the Wasserstein, Kolmogorov and total variation distances have been commonly used in the literature. The the Wasserstein distance
$$
d_{W}(\tilde{X}_{\varepsilon,t},Y_t):=\sup_{|g(x)-g(y)|\le |x-y|}|E g(\tilde{X}_{\varepsilon,t})-E g(Y_t)|
$$
is easy to bound by using the estimate (\ref{EXY}) below. On the other hand, the Kolmogorov distance is dominated by total variation distance.  Hence, we will focus on bounding the total variation distance defined by%\sup_{A }|P(\tilde{X}_{\varepsilon,t}\in A) - P(Y_t\in A)| =Borel subsets $A$ of the real line and all
$$d_{TV}(\tilde{X}_{\varepsilon,t},Y_t):= 
 \frac{1}{2}\sup_{g}|E g(\tilde{X}_{\varepsilon,t})- E g(Y_t)|, $$
where the supremum is running over all measurable functions bounded by $1.$  Our main tools are the techniques of Malliavin calculus. Particularly, we use a general result established in our previous paper \cite{Dung2022} to obtain the optimal estimate of the form
$$d_{TV}(\tilde{X}_{\varepsilon,t},Y_t)\sim O(\varepsilon)\,\,\,\text{as}\,\,\varepsilon\to 0.$$
We also would like to emphasize that our results are able to apply to the class of stochastic Volterra equations with singular kernels.
%Then, we use the techniques of Malliavin calculus to obtain an explicit estimate for $d_{TV}(\tilde{X}_{\varepsilon,t},Y_t).$ Furthermore, we also show that the rate of convergence is optimal order.
%since the sequence $(\tilde{X}_{\varepsilon,t})_{\varepsilon>0}$ satisfies the central limit theorem.

%We observe that, for each $t\in[0,T],$ $Y_t$ is a normal random variable. Thus the sequence $(\tilde{X}_{\varepsilon,t})_{\varepsilon\in (0,1)}$ satisfies the central limit theorem as $\varepsilon\to 0.$ %In fact, in some recent decades, the stochastic Volterra equations have been continued to treated by many authors (see \cite {Zhang2010}). In particular, there are many researches on deviations for stochastic Volterra equations with small noise (see \cite{Nualart2000}). However, there are a lot of properties of the solutions to this class of equations that need studying. In this paper, we focus on bounding the total variation distances
%Our tools are the techniques of Malliavin calculus and the result established recently in \cite{Dung2022}. Furthermore, we also show that the convergence rate is of optimal order by describing the exact asymptotic behavior of these convergence when the parameter noise tends to zero.

The rest of the paper is organized as follows. In Section \ref{8ghjs}, we recall some fundamental concepts of Malliavin calculus and a general estimate for the total variation distance established in our previous paper \cite{Dung2022}. The main results of the paper are stated and proved in Section \ref{fkm}. In Section \ref{tglo}, we provide an application to stochastic Volterra equations with fractional Brownian motion kernel is given to illustrate the theory.
\section{Preliminaries}\label{8ghjs}
In this section, for the reader's convenience, we recall some elements of Malliavin calculus (for more details see \cite{nualartm2}).	We suppose that $(B_t)_{t\in [0, T]}$ is defined on a complete probability space $(\Omega ,\mathcal{F},\mathbb{F},P)$, where $\mathbb{F}= (\mathcal{F}_t)_{t \in [0,T]}$ is a natural filtration generated by the Brownian motion $B$. For $h \in L^2[0, T]$, we denote by $B(h)$ the Wiener integral $$B(h)= \int_0^Th_tdB_t.$$
Let $\mathcal{S}$ denote a dense subset of $L^2(\Omega ,\mathcal{F}, P)$ that consists of smooth random variables of the form
\begin{align}
	F = f(B(h_1),B(h_2),..., B(h_n) ),\label{iik}
\end{align}
where $n \in \mathbb{N}, f \in C_0^\infty(\mathbb{R}^n), h_1, h_2, ..., h_n \in L^2[0,T]$. If $F$ has the form (\ref{iik}), we define its Malliavin derivative as the process $DF:= {D_tF, t\in [0,T]}$ given by
$$D_tF =\sum_{k=1}^{n}\frac{\partial f}{\partial x_k}(B(h_1),B(h_2),..., B(h_n))h_k(t).$$
More generally, for each $k\ge 1,$ we can define the iterated derivative operator on a cylindrical random variable by setting
$$ D_{t_1,...,t_k}^{k}F=D_{t_1}...D_{t_k}F. $$
For any $1 \le p,k< \infty$, we denote by $\mathbb{D}^{k,p}$ the closure of $\mathcal{S}$ with respect to the norm
$$||F||_{k,p}^p:= E|F|^p + E\left[\bigg(\int_0^T|D_sF|^2ds\bigg)^{\frac{p}{2}}\right]+...+E\left[\bigg(\int_0^T...\int_0^T|D^{k}_{t_1,...,t_k}F|^2dt_1...dt_k\bigg)^{\frac{p}{2}}\right].$$
A random variable $F$ is said to be Malliavin differentiable if it belongs to $\mathbb{D}^{1,2}$. An important operator in the Malliavin's calculus theory is the divergence operator $\delta$. It is the adjoint of derivative operator $D.$ The domain of $\delta$ is the set of all functions $u\in L^2(\Omega \times [0,T])$ such that
$$ E|\langle DF,u\rangle _{L^2[0,T]}|\le C(u)\|F\|_{L^2(\Omega)}, $$
where $C(u)$ is some positive constant depending on $u.$ In particular, if $u\in Dom\delta ,$ then $\delta(u)$ is characterized by following duality relationships
\begin{align}
\delta(uF)&=F\delta(u)-\left\langle DF,u\right\rangle_{L^2[0,T]}\label{ct*}\\
E[\left\langle DF,u\right\rangle_{L^2[0,T]}]&=E[F\delta (u)] \mbox{ for any } F\in \mathbb{D}^{1,2}.\label{ct**}
\end{align}
We have the following general estimate for the total variation distance.
\begin{lem} \label{dltq}Let $F_1 \in \mathbb{D}^{2,4}$ be such that  $\|DF_1\|_{L^2[0,T]}>0\,\,a.s.$ Then, for any random variable $F_2\in \mathbb{D}^{1,2}$ and any measurable function $\phi$ bounded by $1,$  we have
\begin{align}
&|E\phi(F_1)-E\phi(F_2)|\notag\\
&\leq C \left(E\|DF_1\|^{-8}_{L^2[0,T]}E\left(\int_0^T\int_0^T|D_\theta D_rF_1|^2d\theta dr\right)^2+(E\|DF_1\|^{-2}_{L^2[0,T]})^2\right)^{\frac{1}{4}}\|F_1-F_2\|_{1,2},\label{uu4}
\end{align}
provided that the expectations exist, where $C$ is an absolute constant.
\end{lem}
\begin{proof}See Theorem 3.1 in our recent paper \cite{Dung2022}.
\end{proof}
\section{The main results}\label{fkm}
In this section, we follow \cite{Zhang2008} to impose the following assumptions. Hereafter we use the notations
\begin{align*}
b'(t,s,x)&=\dfrac{\partial b(t,s,x)}{\partial x};\, \,\sigma'(t,s,x)=\dfrac{\partial \sigma(t,s,x)}{\partial x};\\
b''(t,s,x)&=\dfrac{\partial ^2b(t,s,x)}{\partial x^2};\, \,\sigma''(t,s,x)=\dfrac{\partial^2 \sigma(t,s,x)}{\partial x^2}.
\end{align*}
\begin{assum}\label{assum1}
The coefficients $b,\sigma:[0,T]\times[0,T]\times\R\to \R$ have linear growth, i.e.
$$|b(t,s,x)|+|\sigma(t,s,x)|\leq k_1(t,s)(1+|x|)\mbox{  }\forall x\in\R,0\le s, t\le T,$$
and $b(t,s,x),\sigma(t,s,x)$ are differentiable functions in $x$ with bounded derivatives
$$|b'(t,s,x)|+|\sigma'(t,s,x)|\leq k_2(t,s)\mbox{  }\forall x,y\in\R,0\le s, t\le T,$$
%$$|b(t,s,x)-b(t,s,y)|+|\sigma(t,s,x)-\sigma(t,s,y)|\leq k_2(t,s)|x-y|,\mbox{  }\forall x,y\in\R,0\le s, t\le T,$$
where $k_1(t,s)$ and $k_2(t,s)$ are non-negative measurable functions on $[0,T]\times [0,T]$ such that
\begin{equation}\label{udp}
\sup_{t\in[0,T]} \int_0^{t}\left(k_1^{2\alpha}(t,s)+k_2^{2\beta}(t,s)\right)ds\leq L<\infty
\end{equation}
 for some $L>0$ and $\alpha,\beta>1.$
\end{assum}

\begin{assum}\label{assum2}
$b(t,s,x),\sigma(t,s,x)$ are twice differentiable functions in $x,$ and satisfy
$$|b''(t,s,x)|+|\sigma''(t,s,x)|\leq k_3(t,s)\mbox{  }\forall x\in\R,0\le s, t\le T,$$
where $k_3(t,s)$ is non-negative measurable function on $[0,T]\times [0,T]$ such that
\begin{equation}\label{udpw}
\sup_{t\in[0,T]} \int_0^{t}k_3^{2\gamma}(t,s)ds\leq L<\infty
\end{equation}
 for some $L>0$ and $\gamma>1.$
\end{assum}
The main results of this paper are stated in the following theorems.
\begin{thm}\label{theorem1}
Suppose Assumptions \ref{assum1} and \ref{assum2}.  We consider the stochastic processes $(\tilde{X}_{\varepsilon,t})_{0\le t\le T}$ and $(Y_t)_{0\le t\le T}$ defined by (\ref{hfk3}) and (\ref{hfk2}), respectively. Then, we have
\begin{align*}
d_{TV}(\tilde{X}_{\varepsilon,t},Y_t)&\le  \Bigg(\bigg(\sup_{0\leq u\leq t}\int_0^uk_2^2(u,s)ds\bigg)+\bigg(\sup_{0\leq u\leq t}\int_0^uk_3^2(u,s)ds\bigg)\Bigg)^{\frac{1}{2}}\\
&\times\left( \sup_{0\leq u\leq t}\int_0^{u}k_1^2(u,s)ds\right)^{\frac{1}{2}}\frac{C\varepsilon}{\sqrt{{\rm Var}(Y_t)}},\,\mbox{  }\forall \varepsilon\in(0,1),\mbox{  }0< t\le T,
\end{align*}
where $C$ is a positive constant not depending on $t$ and $\varepsilon.$
\end{thm}
\begin{thm}\label{theorem2}
Suppose Assumptions \ref{assum1} and \ref{assum2}.  We additionally assume that $b''(t,s,x)$ are continuous in $x.$ Then, for any continuous and bounded function $\varphi,$ we have
\begin{align}
\lim\limits_{\varepsilon\to 0}\frac{E\varphi(\tilde{X}_{\varepsilon,t}) -E\varphi(Y_t)}{\varepsilon}=\frac{1}{2\mbox{Var}(Y_t)}E[\varphi(Y_t)\delta(Z_tDY_t)],\,\,\,0< t\le T,\label{ct}
\end{align}
where the stochastic process $(Z_t)_{0\leq t\le T}$ is given by
\begin{align}\label{hfk4}
Z_t=\int_0^tb'(t,s,x_s)Z_sds+\int_0^tb''(t,s,x_s)Y_s^2ds+2\int_0^t\sigma'(t,s,x_s)Y_sdB_s,\mbox{  }t\in [0,T].
\end{align}

\end{thm}
\begin{rem}$(i)$ Note that the It\^o stochastic integral $\int_0^t\sigma(t,s,x_s)dB_s$ is a centered Gaussian random variable with finite variance for each $t\in [0,T].$ Hence, it is easy to verify that the equation (\ref{hfk2}) admits a unique solution and this solution satisfies $E[Y_t]=0$ and for every $p\geq 2,$
\begin{align}\label{EY}
 \sup_{t\in[0,T]}E|Y_t|^p<\infty.
\end{align}

\noindent$(ii)$ For $p\ge \max\left\{\frac{2\beta}{\beta-1},\frac{2\gamma}{\gamma-1}\right\}>2,$ by using the H\"older and Burkholder-Davis-Gundy inequalities, we have
\begin{align*}
&E\left|\int_0^tb''(t,s,x_s)Y_s^2ds+2\int_0^t\sigma'(t,s,x_s)Y_sdB_s\right|^p\\
&\le 2^{p-1}\Bigg(E\left|\int_0^tb''(t,s,x_s)Y_s^2ds\right|^p+E\left|2\int_0^t\sigma'(t,s,x_s)Y_sdB_s\right|^p\Bigg)\\
&\le 2^{p-1}\left(tE\int_0^tk_3^2(t,s)Y_s^4ds\right)^{\frac{p}{2}}+2^{2p-1}\left(E\int_0^tk_2^2(t,s)Y_s^2ds\right)^{\frac{p}{2}}\\
&\le 2^{p-1}t^{\frac{p}{2}}\left(\int_0^tk_3^{\frac{2p}{p-2}}(t,s)ds\right)^{\frac{p}{2}-1}\int_0^tE|Y_s|^{2p}ds+ 2^{2p-1}\left(\int_0^tk_2^{\frac{2p}{p-2}}(t,s)ds\right)^{\frac{p}{2}-1}\int_0^tE|Y_s|^{p}ds.
\end{align*}
This, together with the conditions (\ref{udp}) and (\ref{udpw}), yields
$$\sup_{0\le t\le T}E\left|\int_0^tb''(t,s,x_s)Y_s^2ds+2\int_0^t\sigma'(t,s,x_s)Y_sdB_s\right|^p<\infty.$$
Hence, the linear integral equation (\ref{hfk4}) admits a unique solution $(Z_t)_{0\le t\le T}$ satisfying
$$ \sup_{0\le t\le T}E|Z_t|^p<\infty. $$
\end{rem}

\subsection{Estimates for Malliavin derivatives}
%For any $a,b\in\mathbb{R},$ we denote $a\vee b=\max\left\{a,b\right\}$ and $a\wedge b=\min\left\{a,b\right\}.$

Hereafter, we denote by $C$ a generic constant which may vary at each appearance. In our proofs, we frequently use the fundamental inequality
$$ (a_1+...+a_n)^p\le n^{p-1}(a_1^p+...+a_n^p), $$ for all $a_1,...,a_n\ge 0$ and $p\ge 1.$

%menh de 3.1
\begin{prop}\label{pro3.1} Suppose Assumption \ref{assum1}. Let $(\Xt)_{t\in[0,T]}$ be the solution to the equation (\ref{hfk}). Then, for every $p\ge 2,$ we have
\begin{align}
 \sup_{0\le t\le T}E|\Xt|^p\le C\,\,\,\forall\, \varepsilon\in (0,1),\label{es1}
\end{align}
where $C$ is a positive constant not depending on $t$ and $\varepsilon.$
\end{prop}
\begin{proof}We first recall that, under Assumption \ref{assum1}, there exists a unique solution $(\Xt)_{t\in[0,T]}$ to the equation (\ref{hfk}). In addition, the boundedness of moments is well known, see e.g. \cite{Wang2008}. Here we give a proof to show that the moments are bounded uniformly in $\varepsilon\in (0,1).$ We consider $p\geq \frac{2\alpha}{\alpha-1}>2,$ it follows from  Assumption \ref{assum1} and the H\"older and Burkholder-Davis-Gundy inequalities that
\begin{align*}
E\left|X_{\varepsilon,t}\right|^p&\le 3^{p-1}\left(|x_0|^p+E\left|\int_0^tb(t,s,X_{\varepsilon,s})ds\right|^p+E\left|\varepsilon\int_0^t\sigma(t,s,X_{\varepsilon,s})dB_s \right|^p\right)\\
&\le 3^{p-1}\left(|x_0|^p+E\left(t\int_0^t|b(t,s,X_{\varepsilon,s})|^2ds\right)^\frac{p}{2}+C\varepsilon^pE\left(\int_0^t|\sigma(t,s,X_{\varepsilon,s})|^2ds \right)^{\frac{p}{2}}\right)\\
&\le C+CE\left(\int_0^tk_1^2(t,s)(1+|X_{\varepsilon,s}|^2)ds\right)^\frac{p}{2}\,\,\forall\,\varepsilon\in(0,1),t\in[0,T],
\end{align*}
where $C$ is a constant positive depending only on $T,p$ and $x_0.$  Using the H\"{o}lder inequality we have
\begin{align}
 E\left(\int_0^tk_1^2(t,s)(1+|X_{\varepsilon,s}|^2)ds\right)^\frac{p}{2}
&\le \left( \int_0^tk_1^{\frac{2p}{p-2}}(t,s)ds\right)^{\frac{p}{2}-1}\int_0^tE(1+|X_{\varepsilon,s}|^2)^{\frac{p}{2}}ds\notag\\
&\le 2^{\frac{p}{2}-1}\left( \int_0^tk_1^{\frac{2p}{p-2}}(t,s)ds\right)^{\frac{p}{2}-1}\int_0^tE(1+|X_{\varepsilon,s}|^{p})ds.\notag
\end{align}
Furthermore, it follows from the condition (\ref{udp}) that
$$\int_0^tk_1^{\frac{2p}{p-2}}(t,s)ds\leq t^{1-\frac{p}{(p-2)\alpha}}\bigg(\int_0^tk_1^{2\alpha}(t,s)ds\bigg)^{\frac{p}{(p-2)\alpha}}\leq T^{1-\frac{p}{(p-2)\alpha}}L^{\frac{p}{(p-2)\alpha}}.$$
We deduce
\begin{align}
 E\left(\int_0^tk_1^2(t,s)(1+|X_{\varepsilon,s}|^2)ds\right)^\frac{p}{2}
&\le C+C\int_0^{t}E|X_{\varepsilon,s}|^pds\,\,\forall\,\varepsilon\in(0,1),t\in[0,T],\label{ctB}
\end{align}
where $C$ is a constant positive depending only on $L,T,p$ and $\alpha.$
As a consequence, we obtain
\begin{align*}
 E\left|X_{\varepsilon,t}\right|^p\le C+C\int_0^{t}E|X_{\varepsilon,s}|^pds\,\,\forall\,\varepsilon\in(0,1),t\in[0,T].
\end{align*}
 By Gronwall's lemma, we get
$$E|\Xt|^p\le Ce^{Ct}\leq Ce^{CT}\,\,\forall\,\varepsilon\in(0,1),t\in[0,T].$$
So, by Lyapunov's inequality, we conclude that (\ref{es1}) holds true for any $p\geq 2.$ The proof of the proposition is complete.
\end{proof}

%menh de 3.2

%menh de 3.3
\begin{prop}\label{pro3.3}
 Let Assumption \ref{assum1} hold. Let $(\Xt)_{t\in[0,T]}$  be the solution to the equation (\ref{hfk}). Then, for each $0\le t\le T,$ the random variable $\Xt$ is Malliavin differentiable. Moreover, the derivative $D_{\theta}\Xt$ satisfies  $D_{\theta}\Xt =0$ for $\theta>t$ and
\begin{multline}
 D_{\theta}\Xt =\varepsilon \sigma(t,\theta,X_{\varepsilon,\theta})+\int_{\theta}^tb'(t,s,X_{\varepsilon,s})D_{\theta}\Xs ds\\+\varepsilon\int_{\theta}^t\sigma'(t,s,X_{\varepsilon,s})D_{\theta}\Xs dB_s,\mbox{  }0\le \theta\le t\le T.\label{DX}
\end{multline}
\end{prop}
\begin{proof}
%Under Assumption \ref{assum2}, the functions $b(t,s,x)$ and $\sigma(t,s,x)$ are Lipschitz functions and have linear growth in $x$. Indeed, there exist positive constant $M,L$ such that
%\begin{align}\label{as1}
 %|b(t,s,x)-b(t,s,y)|+|\sigma(t,s,x)-\sigma(t,s,y)|\le M|x-y|, \mbox{  } \forall x,y\in \R,\mbox{  }t,s\in [0,T],
%\end{align}
%and
%$$|\sigma(t,s,x)|\le|\sigma(t,s,x_0)|+|\sigma(t,s,x)-\sigma(t,s,x_0)|\le C+M|x-x_0|\le L(1+|x|),\mbox{  }\forall x\in\R,\mbox{  }t,s\in [0,T].$$
Consider the Picard approximations given by
\begin{align*}
X^{(0)}_{\varepsilon,t}&=x_0\\
X^{(n)}_{\varepsilon,t}&=x_0+\int_0^tb(t,s,X^{(n-1)}_{\varepsilon,s})ds+\varepsilon\int_0^t\sigma(t,s,X^{(n-1)}_{\varepsilon,s})dB_s,\mbox{  }n\ge 1.
\end{align*}
It is easy to verify that $  \sup_{n\geq 0}\sup_{0\le t\le T}E|X^{(n)}_{\varepsilon,t}|^p<\infty\,\,\forall\, \varepsilon\in (0,1)$ and $p\geq 2. $ Moreover,
$$X^{(n)}_{\varepsilon,t}\to X_{\varepsilon,t}\,\,\,\text{in}\,\,L^p(\Omega)\,\,\text{as}\,\,n\to\infty.$$
(The proof is similar to that of Theorem 1.1 in \cite{Wang2008}). Fixed $t\in [0,T],$ by the induction argument, we can show that $X^{(n)}_{\varepsilon,t} $ belongs to $\mathbb{D}^{1,2}$ for all $n\ge 0.$ Moreover, $  D_{\theta}X^{(n)}_{\varepsilon,t}=0 \mbox{ for } \theta>t,$ and for $0\le \theta\le t\le T,$ we have
\begin{align*}
  D_{\theta}X^{(n)}_{\varepsilon,t}&= \varepsilon \sigma(t,\theta,X^{(n-1)}_{\varepsilon,\theta})+\int_{\theta}^tD_\theta[b(t,s,X^{(n-1)}_{\varepsilon,s})] ds+\varepsilon\int_{\theta}^tD_\theta[\sigma(t,s,X^{(n-1)}_{\varepsilon,s})]dB_s\\
&=\varepsilon \sigma(t,\theta,X^{(n-1)}_{\varepsilon,\theta})+\int_{\theta}^t b'(t,s,X^{(n-1)}_{\varepsilon,s})D_\theta X^{(n-1)}_{\varepsilon,s}ds+\varepsilon\int_{\theta}^t \sigma'(t,s,X^{(n-1)}_{\varepsilon,s})D_\theta X^{(n-1)}_{\varepsilon,s}dB_s.
\end{align*}
We now claim that
\begin{equation}\label{7gt}
\sup\limits_n E\|DX^{(n)}_{\varepsilon,t}\|^2_{L^2[0,T]}<\infty,\mbox{  }0\le t\le T.
\end{equation}
We have
\begin{align*}
E|  D_{\theta}X^{(n)}_{\varepsilon,t}|^2&\le 3\Bigg[E|\varepsilon \sigma(t,\theta,X^{(n-1)}_{\varepsilon,\theta})|^{2}+E\left|\int_{\theta}^t b'(t,s,X^{(n-1)}_{\varepsilon,s})D_\theta X^{(n-1)}_{\varepsilon,s}ds\right|^{2}\\
&+E\left|\varepsilon\int_{\theta}^t \sigma'(t,s,X^{(n-1)}_{\varepsilon,s})D_\theta X^{(n-1)}_{\varepsilon,s}dB_s\right|^{2}\Bigg]\\
&\le 3E|\varepsilon \sigma(t,\theta,X^{(n-1)}_{\varepsilon,\theta})|^{2}+3t\int_{\theta}^tE|b'(t,s,X^{(n-1)}_{\varepsilon,s})D_\theta X^{(n-1)}_{\varepsilon,s}|^{2}ds\\
&+3\varepsilon^2\int_{\theta}^tE|\sigma'(t,s,X^{(n-1)}_{\varepsilon,s})D_\theta X^{(n-1)}_{\varepsilon,s}|^2ds,\,\,0\le \theta\le t\le T.
\end{align*}
By Assumption \ref{assum1}, we get
\begin{align*}
E|  D_{\theta}X^{(n)}_{\varepsilon,t}|^2&\le 6\varepsilon^2 k_1^2(t,\theta)(1+E|X^{(n-1)}_{\varepsilon,\theta}|^2)+3(T+\varepsilon^2)\int_{\theta}^tk_2^2(t,s)E|D_\theta X^{(n-1)}_{\varepsilon,s}|^2ds\\
&\le C\varepsilon^2 k_1^2(t,\theta)+C\int_{\theta}^tk_2^2(t,s)E|D_\theta X^{(n-1)}_{\varepsilon,s}|^2ds,\,\,0\le \theta\le t\le T,
\end{align*}
where $C$ is the positive constant not depending on $n,\varepsilon\mbox{ and }t$. We obtain
$$\int_0^t E| D_{\theta}X^{(n)}_{\varepsilon,t} |^2d\theta\le C\varepsilon^2\int_0^t k_1^2(t,\theta)d\theta+C\int_0^t k_2^2(t,s)\int_0^sE|D_\theta X^{(n-1)}_{\varepsilon,s}|^2d\theta ds,\,\,\,0\le t\le T.$$
Fixed $p\ge \frac{2\beta}{\beta-1}>2$, by using the same arguments as in the proof of (\ref{ctB}), we get
\begin{align*}
\left(\int_0^{t} E| D_{\theta}X^{(n)}_{\varepsilon,t} |^2d\theta\right)^{\frac{p}{2}}
&\le 2^{\frac{p}{2}-1}\Bigg(\left(\int_0^{t}C\varepsilon^2k_1^2(t,\theta)d\theta\right)^{\frac{p}{2}}+\left( C\int_0^tk_2^2(t,s)\int_0^s |D_\theta X^{(n-1)}_{\varepsilon,s}|^2d\theta ds\right)^{\frac{p}{2}}\Bigg)\\
%&\le C\varepsilon^p+C\left(\int_0^tk_3^2(t,s)\int_0^s\left|D_\theta X_{\varepsilon,s}\right|^2d\theta ds\right)^{\frac{p}{2}}\\
%&\le C\varepsilon^p+C\left(\int_0^tk_3^{\frac{2p}{p-2}}(t,s)ds\right)^{\frac{p}{2}-1}\int_0^t\left(\int_0^s\left|D_\theta X_{\varepsilon,s}\right|^2d\theta\right)^{\frac{p}{2}}ds\\
&\le C\left( \sup_{0\leq u\leq t}\int_0^{u}k_1^2(u,\theta)d\theta\right)^{\frac{p}{2}}\varepsilon^p +C\int_0^t\left(\int_0^s|D_\theta X^{(n-1)}_{\varepsilon,s}|^2d\theta\right)^{\frac{p}{2}}ds.
\end{align*}
As a consequence
\begin{multline}\label{ct10}
\left(E\|DX^{(n-1)}_{\varepsilon,t}\|^2_{L^2[0,T]}\right)^{\frac{p}{2}}\le C\left( \sup_{0\leq u\leq t}\int_0^{u}k_1^2(u,\theta)d\theta\right)^{\frac{p}{2}}\varepsilon^p\\+C\int_0^{t}\left(E\|DX^{(n-1)}_{\varepsilon,s}\|^2_{L^2[0,T]}\right)^{\frac{p}{2}}ds,\,\,0\le t\le T.
\end{multline}
Put $u(t)=\sup\limits_n \left(E\|DX^{(n)}_{\varepsilon,t}\|^2_{L^2[0,T]}\right)^{\frac{p}{2}}.$ Taking supremum two side of (\ref{ct10}) we get
$$
u(t)\le C\left( \sup_{0\leq u\leq t}\int_0^{u}k_1^2(u,\theta)d\theta\right)^{\frac{p}{2}}\varepsilon^{p}+C\int_0^t u(s)ds,\mbox{  }0\le t\le T.
$$
Then, using Gronwall's lemma, we obtain
\begin{equation}\label{ser}u(t)\le C\left( \sup_{0\leq u\leq t}\int_0^{u}k_1^2(u,\theta)d\theta\right)^{\frac{p}{2}}\varepsilon^{p}e^{Ct}<\infty,\mbox{  }0\le t\le T,\mbox{  }\varepsilon\in (0,1).
\end{equation}
This implies that the claim (\ref{7gt}) holds true. So, in view of Lemma 1.5.3 in \cite{nualartm2}, we conclude that for any $t\in [0,T], X_{\varepsilon,t}$ is Malliavin differentiable and its derivative satisfies the equation (\ref{DX}).

The proof of the proposition is complete.
%\begin{align}
 %D_{\theta}\Xt &=0 \mbox{ for }\theta>t \mbox{ and }\notag\\
% D_{\theta}\Xt &=\varepsilon \sigma(t,\theta,X_{\varepsilon,\theta})+\int_{\theta}^tb'(t,s,X_{\varepsilon,s})D_{\theta}\Xs ds+\varepsilon\int_{\theta}^t\sigma'(t,s,X_{\varepsilon,s})D_{\theta}\Xs dB_s,\mbox{  }0\le \theta\le t\le T.\notag
%\end{align}
\end{proof}

%menh de 3.4
\begin{prop}\label{pro3.4}Suppose Assumption \ref{assum1}. Let $(\Xt)_{t\in[0,T]}$  be the solution to the equation (\ref{hfk}).  Then, we have
\begin{align}\label{DX1}
\int_0^{t} E| D_{\theta}\Xt |^2d\theta\le C\left( \sup_{0\leq u\leq t}\int_0^{u}k_1^2(u,\theta)d\theta\right)\varepsilon^{2}\, \mbox{  }\forall \varepsilon\in (0,1),
\end{align}
where $C$ is a positive constant not depending on $t\mbox{ and }\varepsilon.$
\end{prop}
\begin{proof}The proof is similar to that of (\ref{ser}). So we omit it.
\end{proof}

%\begin{rem}
%(i)
%(ii) Note that $D_{\theta}Y_t$ is deterministic for all $0\le \theta\le t\le T.$ Hence, by Clark-Ocone formulla, we have
%$$ Y_t=\int_0^tE[D_\theta Y_s|\mathcal{F}_\theta]dB_\theta=\int_0^t D_\theta Y_s dB_\theta.$$
%This representation formula shows that $Y_t$ is a normal random variable for each $t\in (0,T].$ Moreover, we have $E[Y_t]=0$ and $\mbox{Var}(Y_t)=\|DY_t\|^2_{L^2[0,T]}$ for all $0\le t\le T.$
%\end{rem}
\subsection{Proof of Theorem \ref{theorem1}}
The proof of Theorem \ref{theorem1} will be given at the end of this subsection. Let us first prepare some technical results.
\begin{prop}\label{pro3.2}
 Let Assumption \ref{assum1} hold. Let $(\Xt)_{t\in[0,T]}$ and $(x_t)_{t\in[0,T]}$ be the solution to the equations (\ref{hfk}) and (\ref{hfk1}), respectively. Then, for all $p\ge 2$ we have
\begin{equation}\label{EX1}
E|X_{\varepsilon,t}-x_t|^p \le C\left( \sup_{0\leq u\leq t}\int_0^{u}k_1^2(u,s)ds\right)^{\frac{p}{2}}\varepsilon^p\,\,\forall\, \varepsilon\in (0,1), 0\le t\le T,
\end{equation}
where $C$ is a positive constant not depending on $t$ and $\varepsilon.$
\end{prop}
\begin{proof}
For every $\varepsilon\in (0,1),$ we have
$$X_{\varepsilon,t}-x_t=\int_0^t\left(b(t,s,X_{\varepsilon,s})-b(t,s,x_s)\right)ds+\varepsilon \int_0^t\sigma(t,s,X_{\varepsilon,s})dB_s,\mbox{  }0\le t\le T. $$
For every $p\geq \frac{2\beta}{\beta-1}> 2,$ using  the H\"older and Burkholder-Davis-Gundy inequalities and Assumption \ref{assum1}, we get
\begin{align}
E&|X_{\varepsilon,t}-x_t|^p\le 2^{p-1}\left(E\left|\int_0^t\left(b(t,s,X_{\varepsilon,s})-b(t,s,x_s)\right)ds\right|^p+\varepsilon^pE\left| \int_0^t\sigma(t,s,X_{\varepsilon,s})dB_s\right|^p\right)\notag\\
&\le CE\left(t\int_0^tk_2^2(t,s)|X_{\varepsilon,s}-x_s|^2ds\right)^{\frac{p}{2}}+C\varepsilon^pE\left(\int_0^{t}k_1^2(t,s)(1+|\Xs|^2)ds\right)^{\frac{p}{2}},\label{ct4}
\end{align}
where $C$ is a positive constant depending only on $p.$ Using (\ref{es1}), we have the following estimate for the second addend in the right hand of (\ref{ct4})
\begin{align*}
E\left(\int_0^{t}k_1^2(t,s)(1+|\Xs|^2)ds\right)^{\frac{p}{2}}&\le \left(\int_0^{t}k_1^2(t,s)ds\right)^{\frac{p}{2}-1}\left(\int_0^{t}k_1^2(t,s)E(1+|\Xs|^2)^{\frac{p}{2}}ds\right)\\
&\leq C\left(\int_0^{t}k_1^2(t,s)ds\right)^{\frac{p}{2}},\,\,0\leq t\leq T,
\end{align*}
where $C$ is a positive constant not depending on $t$ and $\varepsilon.$ For the first term in the right hand of (\ref{ct4}), we use the same arguments as in the proof of (\ref{ctB}) to obtain
\begin{align}
E\left(\int_0^tk_2(t,s)|X_{\varepsilon,s}-x_s|ds\right)^p
&\le E\left(t\int_0^tk_2^2(t,s)|X_{\varepsilon,s}-x_s|^2ds\right)^{\frac{p}{2}}\notag\\
&\le C\left(\int_0^tk_2^{\frac{2p}{p-2}}(t,s)ds\right)^{\frac{p}{2}-1}\int_0^tE|X_{\varepsilon,s}-x_s|^pds\notag\\
&\le C\int_0^tE|X_{\varepsilon,s}-x_s|^pds.\label{ct11}
\end{align}
The above estimates imply that
\begin{align*}
E|X_{\varepsilon,t}-x_t|^p\le C\left(\int_0^{t}k_1^2(t,s)ds\right)^{\frac{p}{2}}\varepsilon^p+C\int_0^tE|X_{\varepsilon,s}-x_s|^pds, \mbox{  }0\le t\le T,
\end{align*}
where $C$ is a positive constant depending only on $p,T,\beta.$
By Gronwall's lemma, we get
$$ E|X_{\varepsilon,t}-x_t|^p\le C\left( \sup_{0\leq u\leq t}\int_0^{u}k_1^2(u,s)ds\right)^{\frac{p}{2}}\varepsilon^p e^{Ct}\le C\left( \sup_{0\leq u\leq t}\int_0^{u}k_1^2(u,s)ds\right)^{\frac{p}{2}}\varepsilon^p,\mbox{  }0\le t\le T.$$
So, by Lyapunov's inequality, the estimate  (\ref{EX1}) holds true for any $p\geq 2.$ This completes the proof of the proposition.
\end{proof}
\begin{prop}\label{pro1}
Suppose Assumption \ref{assum1} and Assumption \ref{assum2}. Consider the stochastic processes $(\tilde{X}_{\varepsilon,t})_{0\le t\le T}$ and $(Y_t)_{0\le t\le T}$ defined by (\ref{hfk3}) and (\ref{hfk2}), respectively. For every $p\ge 2,$ we have
\begin{multline}\label{EXY}
E|\tilde{X}_{\varepsilon,t}-Y_t|^p\le C \left(\bigg(\sup_{0\leq u\leq t}\int_0^uk_2^2(u,s)ds\bigg)^{\frac{p}{2}} +\bigg(\sup_{0\leq u\leq t}\int_0^uk_3^2(u,s)ds\bigg)^{\frac{p}{2}}\right)\\\times\left( \sup_{0\leq u\leq t}\int_0^{u}k_1^2(u,s)ds\right)^{\frac{p}{2}}\varepsilon^{p}\, \,\forall \varepsilon\in (0,1), 0\le t\le T,
\end{multline}
where $C$ is a positive constant not depending on $t$ and $\varepsilon.$
\end{prop}
\begin{proof}
For every $\varepsilon\in (0,1),$ we have
\begin{align}
\tilde{X}_{\varepsilon,t}-Y_t=\frac{1}{\varepsilon}\int_0^t& \left(b(t,s,\Xs)-b(t,s,x_s)\right)ds -\int_0^tb'(t,s,x_s)Y_sds\notag\\
&+\int_0^t\left(\sigma(t,s,\Xs)-\sigma(t,s,x_s)\right)dB_s,\,\, 0\le t\le T.\label{ct1}
\end{align}
For each $s\in[0,T],$ using Taylor's expansion, we have
\begin{align*}
b(t,s,\Xs)-b(t,s,x_s)=b'(t,s,x_s)(\Xs-x_s)+ \frac{1}{2}b''(t,s,x_s+\xi _1(\Xs-x_s))(\Xs-x_s)^2,
\end{align*}
where $\xi_1$ is random variable lying between 0 and 1. We can rewrite (\ref{ct1}) as follows
\begin{align}
\tilde{X}_{\varepsilon,t}-Y_t= \int_0^t b'(t,s,x_s)&\big(\tilde{X}_{\varepsilon,s}-Y_s\big)ds+\frac{1}{2\varepsilon}\int_0^t b''(t,s,x_s+\xi _1(\Xs-x_s))(\Xs-x_s)^2ds\notag\\
&+\int_0^t\left(\sigma(t,s,\Xs)-\sigma(t,s,x_s)\right)dB_s,\,\, 0\le t\le T. \label{ct3}
\end{align}
Hence, for every $p\ge 2,$ we get
\begin{align*}
E|\tilde{X}_{\varepsilon,t}-Y_t|^p&\le 3^{p-1}\bigg(E\left|\int_0^t b'(t,s,x_s)\big(\tilde{X}_{\varepsilon,s}-Y_s\big)ds\right|^p\\
&+\frac{1}{2^{p}\varepsilon^p}E\left|\int_0^t b''(t,s,x_s+\xi _1(\Xs-x_s))(\Xs-x_s)^2ds\right|^p\\
&+E\left|\int_0^t\left(\sigma(t,s,\Xs)-\sigma(t,s,x_s)\right)dB_s\right|^p\bigg),\,\, 0\le t\le T.
\end{align*}
Then, by using the H\"{o}lder and Burkholder-Davis-Gundy inequalities we deduce
\begin{align*}
E|\tilde{X}_{\varepsilon,t}-Y_t|^p&\le 3^{p-1}\Bigg(E\left(t\int_0^tk_2^2(t,s)|\tilde{X}_{\varepsilon,s}-Y_s|^2ds\right)^{\frac{p}{2}}+\frac{1}{2^{p}\varepsilon^p}E\left(t\int_0^tk_3^2(t,s)\left|\Xs-x_s\right|^4ds\right)^{\frac{p}{2}}\\
&+C\left(E\int_0^tk_2^2(t,s)|\Xs-x_s|^{2}ds\right)^{\frac{p}{2}}\Bigg),\\
&\le CE\left(\int_0^tk_2^2(t,s)|\tilde{X}_{\varepsilon,s}-Y_s|^2ds\right)^{\frac{p}{2}}+\frac{C}{\varepsilon^p}E\left(\int_0^tk_3^2(t,s)\left|\Xs-x_s\right|^4ds\right)^{\frac{p}{2}}\\
&+CE\left(\int_0^tk_2^2(t,s)|\Xs-x_s|^{2}ds\right)^{\frac{p}{2}},\,\, 0\le t\le T,
\end{align*}
where $C$ is a positive constant depending only on $p$ and $T.$ We now consider $p\ge \frac{2\beta}{\beta-1}>2.$ By using the H\"{o}lder inequality and the estimate (\ref{EX1}) we get
\begin{align*}
E\left(\int_0^tk_2^2(t,s)|\tilde{X}_{\varepsilon,s}-Y_s|^2ds\right)^{\frac{p}{2}}&\leq \left(\int_0^tk_2^\frac{2p}{p-2}(t,s)ds\right)^{\frac{p}{2}-1}\int_0^t E|\tilde{X}_{\varepsilon,s}-Y_s|^pds\\
&\leq C\int_0^t E|\tilde{X}_{\varepsilon,s}-Y_s|^pds,
\end{align*}
\begin{align*}
E\left(\int_0^tk_3^2(t,s)\left|\Xs-x_s\right|^4ds\right)^{\frac{p}{2}}&\leq \left(\int_0^tk_3^2(t,s)ds\right)^{\frac{p}{2}-1}\int_0^tk_3^2(t,s)\left|\Xs-x_s\right|^{2p}ds\\
&\leq \left(\int_0^tk_3^2(t,s)ds\right)^{\frac{p}{2}}\left( \sup_{0\leq u\leq t}\int_0^{u}k_1^2(u,s)ds\right)^{p}\varepsilon^{2p},
\end{align*}
and
\begin{align*}
E\left(\int_0^tk_2^2(t,s)\left|\Xs-x_s\right|^2ds\right)^{\frac{p}{2}}&\leq \left(\int_0^tk_2^2(t,s)ds\right)^{\frac{p}{2}-1}\int_0^tk_2^2(t,s)\left|\Xs-x_s\right|^{p}ds\\
&\leq \left(\int_0^tk_2^2(t,s)ds\right)^{\frac{p}{2}}\left( \sup_{0\leq u\leq t}\int_0^{u}k_1^2(u,s)ds\right)^{\frac{p}{2}}\varepsilon^{p}.
\end{align*}
So it holds that
\begin{align*}
E|&\tilde{X}_{\varepsilon,t}-Y_t|^p\leq C\int_0^t E|\tilde{X}_{\varepsilon,s}-Y_s|^pds\\
&+ C \left(\bigg(\sup_{0\leq u\leq t}\int_0^uk_2^2(u,s)ds\bigg)^{\frac{p}{2}} +\bigg(\sup_{0\leq u\leq t}\int_0^uk_3^2(u,s)ds\bigg)^{\frac{p}{2}}\right)\left( \sup_{0\leq u\leq t}\int_0^{u}k_1^2(u,s)ds\right)^{\frac{p}{2}}\varepsilon^{p},
\end{align*}
which, by Gronwall's lemma, gives us the desired conclusion  (\ref{EXY}). This completes the proof of the proposition.
%$$ E|\tilde{X}_{\varepsilon,t}-Y_t|^p\le C\varepsilon^{p}e^{Ct}\le C\varepsilon^{p}, $$
%where $C$ is positive constant depending on $T,p.$ Lyapunov's inequality give us thatis true for any $p\geq 2.$
\end{proof}

%meng de 3.6
\begin{prop}\label{pro2}
Suppose Assumptions \ref{assum1} and \ref{assum2}. Consider the stochastic processes $(\tilde{X}_{\varepsilon,t})_{0\le t\le T}$ and $(Y_t)_{0\le t\le T}$ defined by (\ref{hfk3}) and (\ref{hfk2}). Then we have
\begin{multline}\label{DXY}
E\|D\tilde{X}_{\varepsilon,t}-DY_t\|^2_{L^2[0,T]}\le C \left(\bigg(\sup_{0\leq u\leq t}\int_0^uk_2^2(u,s)ds\bigg)+\bigg(\sup_{0\leq u\leq t}\int_0^uk_3^2(u,s)ds\bigg)\right)\\\times\left( \sup_{0\leq u\leq t}\int_0^{u}k_1^2(u,s)ds\right)\varepsilon^{2}\, \,\forall \varepsilon\in (0,1), 0\le t\le T,
\end{multline}
where $C$ is a positive constant not depending on $t$ and $\varepsilon.$
\end{prop}
\begin{proof}For each $t\in[0,T],$ the random variable $Y_t$ is Malliavin differentiable and its derivative is given by $D_{\theta}Y_t=0$ for $\theta>t$ and
\begin{align}\label{DY}
 D_{\theta}Y_t=\sigma(t,\theta, x_\theta)+\int_{\theta}^{t}b'(t,s,x_s)D_{\theta}Y_s ds,\mbox{  }0\le \theta\le t\le T.
\end{align}
By the definition of $\tilde{X}_{\varepsilon,t}$, we have $D_{\theta}\tilde{X}_{\varepsilon,t}=\frac{1}{\varepsilon}D_{\theta}{X}_{\varepsilon,t}.$ Hence, from (\ref{DX}) and (\ref{DY}), we have
\begin{align*}
&D_{\theta}\tilde{X}_{\varepsilon,t}-D_{\theta}Y_t=\sigma(t,\theta,X_{\varepsilon,\theta})-\sigma(t,\theta, x_\theta)\\
&+\int_{\theta}^tb'(t,s,X_{\varepsilon,s})D_{\theta}\tilde{X}_{\varepsilon,s}ds-\int_{\theta}^{t}b'(t,s,x_s)D_{\theta}Y_s ds+\int_{\theta}^t\sigma'(t,s,X_{\varepsilon,s})D_{\theta}\Xs dB_s,\,\,0\le \theta\leq t\le T.
\end{align*}
Consequently, we deduce
\begin{align}
&E|D_{\theta}\tilde{X}_{\varepsilon,t}-D_{\theta}Y_t|^2\le 3E\left|\sigma(t,\theta,X_{\varepsilon,\theta})-\sigma(t,\theta, x_\theta)\right|^2\notag\\
&+3E\left|\int_{\theta}^t(b'(t,s,X_{\varepsilon,s})D_{\theta}\tilde{X}_{\varepsilon,s}-b'(t,s,x_s)D_{\theta}Y_s ) ds\right|^2+3E\left|\int_{\theta}^t\sigma'(t,s,X_{\varepsilon,s})D_{\theta}\Xs dB_s\right|^2\label{ct2}.
\end{align}
For the first term in the right hand side of (\ref{ct2}), we use the estimate (\ref{EX1}) to get
\begin{align*}
3E\left|\sigma(t,\theta,X_{\varepsilon,\theta})-\sigma(t,\theta, x_\theta)\right|^2&\le 3k_2^2(t,\theta)E|X_{\varepsilon,\theta}-x_\theta|^2\\
&\le Ck_2^2(t,\theta)\left( \sup_{0\leq u\leq t}\int_0^{u}k_1^2(u,s)\right)\varepsilon^2,\,\,0\le \theta\leq t\le T,
\end{align*}
where $C$ is a positive constant not depending on $t$ and $\varepsilon.$ For the second term, we use the Cauchy-Schwart inequality and Assumptions \ref{assum1} and \ref{assum2} to get
\begin{align*}
&3E\left|\int_{\theta}^t(b'(t,s,X_{\varepsilon,s})D_{\theta}\tilde{X}_{\varepsilon,s}-b'(t,s,x_s)D_{\theta}Y_s) ds\right|^2\\
&\le 6E\left|\int_{\theta}^t\left(b'(t,s,X_{\varepsilon,s})-b'(t,s,x_s)\right)D_{\theta}Y_sds\right|^2
+6E\left|\int_{\theta}^tb'(t,s,X_{\varepsilon,s})(D_{\theta}\tilde{X}_{\varepsilon,s}-D_{\theta}Y_s)ds\right|^2\\
&\le 6t\int_{\theta}^tk_3^2(t,s)E\left|X_{\varepsilon,s}-x_s\right|^2ds
+6t\int_{\theta}^tk_2^2(t,s)E|D_{\theta}\tilde{X}_{\varepsilon,s}-D_{\theta}Y_s|^2ds\\
&\le C\left( \sup_{0\leq u\leq t}\int_0^{u}k_1^2(u,s)ds\right)\varepsilon^2\int_{\theta}^tk_3^2(t,s)ds
+C\int_{\theta}^tk_2^2(t,s)E|D_{\theta}\tilde{X}_{\varepsilon,s}-D_{\theta}Y_s|^2ds,
\end{align*}
Furthermore, by the Itô isometry and the estimates (\ref{DX1}), we can estimates the last term in the right hand side of (\ref{ct2}) as follows
\begin{align*}
3E\left|\int_{\theta}^t\sigma'(t,s,X_{\varepsilon,s})D_{\theta}\Xs dB_s\right|^2&=3\int_{\theta}^tE\left|\sigma'(t,s,X_{\varepsilon,s})D_{\theta}\Xs\right|^2ds\\
&\le C\int_{\theta}^tk_2^2(t,s)E|D_{\theta}\Xs|^2ds.
\end{align*}
Combining the above estimates, we obtain
\begin{align*}
E|D_{\theta}&\tilde{X}_{\varepsilon,t}-D_{\theta}Y_t|^2
\le C\left( \sup_{0\leq u\leq t}\int_0^{u}k_1^2(u,s)ds\right)\varepsilon^2\left(k_2^2(t,\theta)+\int_{\theta}^tk_3^2(t,s)ds\right)\\
&+ C\int_{\theta}^tk_2^2(t,s)E|D_{\theta}\tilde{X}_{\varepsilon,s}-D_{\theta}Y_s|^2ds+C\int_{\theta}^tk_2^2(t,s)E|D_{\theta}\Xs|^2ds, \mbox{  }0\le \theta\le t\le T,
\end{align*}
and hence,
\begin{align*}
&E\|D\tilde{X}_{\varepsilon,t}-DY_t\|^2_{L^2[0,T]}=\int_0^tE|D_{\theta}\tilde{X}_{\varepsilon,t}-D_{\theta}Y_t|^2d\theta\\
&\leq C \Bigg(\bigg(\sup_{0\leq u\leq t}\int_0^uk_2^2(u,s)ds\bigg) +\bigg(\sup_{0\leq u\leq t}\int_0^uk_3^2(u,s)ds\bigg)\Bigg)\left( \sup_{0\leq u\leq t}\int_0^{u}k_1^2(u,s)ds\right)\varepsilon^2\\
&+ C\int_0^tk_2^2(t,s)\int_0^sE|D_{\theta}\tilde{X}_{\varepsilon,s}-D_{\theta}Y_s|^2d\theta ds+C\int_0^tk_2^2(t,s)\int_0^sE|D_{\theta}\Xs|^2d\theta ds.
\end{align*}
Since the estimate (\ref{DX1}) we obtain
\begin{align*}
&E\|D\tilde{X}_{\varepsilon,t}-DY_t\|^2_{L^2[0,T]}\leq C\int_0^tk_2^2(t,s)E\|D\tilde{X}_{\varepsilon,s}-DY_s\|^2_{L^2[0,T]} ds\\
&+C \Bigg(\bigg(\sup_{0\leq u\leq t}\int_0^uk_2^2(u,s)ds\bigg)+\bigg(\sup_{0\leq u\leq t}\int_0^uk_3^2(u,s)ds\bigg)\Bigg)\left( \sup_{0\leq u\leq t}\int_0^{u}k_1^2(u,s)ds\right)\varepsilon^2
\end{align*}
We now consider $p\ge \frac{2\beta}{\beta-1}>2.$
\begin{align*}
&\big(E\|D\tilde{X}_{\varepsilon,t}-DY_t\|^2_{L^2[0,T]}\big)^{\frac{p}{2}}\le C\left(\int_0^tk_2^{\frac{2p}{p-2}}(t,s)ds\right)^{\frac{p}{2}-1}\int_0^t\big(E\|D\tilde{X}_{\varepsilon,s}-DY_s\|^2_{L^2[0,T]}\big)^{\frac{p}{2}}ds\\
&+C \Bigg(\bigg(\sup_{0\leq u\leq t}\int_0^uk_2^2(u,s)ds\bigg)+\bigg(\sup_{0\leq u\leq t}\int_0^uk_3^2(u,s)ds\bigg)\Bigg)^{\frac{p}{2}}\left( \sup_{0\leq u\leq t}\int_0^{u}k_1^2(u,s)ds\right)^{\frac{p}{2}}\varepsilon^{p}\\
&\le C \int_0^t\big(E\|D\tilde{X}_{\varepsilon,s}-DY_s\|^2_{L^2[0,T]}\big)^{\frac{p}{2}}ds\\
&+C \Bigg(\bigg(\sup_{0\leq u\leq t}\int_0^uk_2^2(u,s)ds\bigg)+\bigg(\sup_{0\leq u\leq t}\int_0^uk_3^2(u,s)ds\bigg)\Bigg)^{\frac{p}{2}}\left( \sup_{0\leq u\leq t}\int_0^{u}k_1^2(u,s)ds\right)^{\frac{p}{2}}\varepsilon^{p}.
\end{align*}
where $C$ is a positive constant not depending on $t$ and $\varepsilon.$ An application of Gronwall's lemma give us
\begin{align*}
\big(E\|D\tilde{X}_{\varepsilon,t}-DY_t\|^2_{L^2[0,T]}\big)^{\frac{p}{2}}\le &C \Bigg(\bigg(\sup_{0\leq u\leq t}\int_0^uk_2^2(u,s)ds\bigg)+\bigg(\sup_{0\leq u\leq t}\int_0^uk_3^2(u,s)ds\bigg)\Bigg)^{\frac{p}{2}}\\
&\times \left( \sup_{0\leq u\leq t}\int_0^{u}k_1^2(u,s)ds\right)^{\frac{p}{2}}\varepsilon^{p}.
\end{align*}
This implies (\ref{DXY}). The proof of the proposition is completed.
\end{proof}

\begin{proof}[The proof of Theorem \ref{theorem1}]
Fix $\varepsilon\in (0,1)$ and $t\in (0,T].$ We consider the random variables $F_1=Y_t$ and $F_2=\tilde{X}_{\varepsilon,t}.$ Thanks to Proposition \ref{pro1} and Proposition \ref{pro2} we have
\begin{align*}
\|F_1-F_2\|_{1,2}&=\|\tilde{X}_{\varepsilon,t}-Y_t\|_{1,2}\\
&=\left(E|\tilde{X}_{\varepsilon,t}-Y_t|^2+E\|D\tilde{X}_{\varepsilon,t}-DY_t\|^2_{L^2[0,T]}\right)^{\frac{1}{2}}\\
&\le C\Bigg(\bigg(\sup_{0\leq u\leq t}\int_0^uk_2^2(u,s)ds\bigg)+\bigg(\sup_{0\leq u\leq t}\int_0^uk_3^2(u,s)ds\bigg)\Bigg)^{\frac{1}{2}}\left( \sup_{0\leq u\leq t}\int_0^{u}k_1^2(u,s)ds\right)^{\frac{1}{2}}\varepsilon.
\end{align*}
We observe from (\ref{DY}) that $D_{\theta}Y_t$ are deterministic for all $0\le\theta\le t\le T.$ So $D_rD_\theta Y_t=0,\mbox{  }0\le r,\theta\le t\le T.$ Moreover, by the Clark-Ocone formula,
$$Y_t=EY_t+\int_0^t E[D_{\theta}Y_t|\mathcal{F}_\theta]dB_\theta=\int_0^t D_{\theta}Y_tdB_\theta,\,\,0\leq t\leq T,$$
and hence,
$$ \|DF_1\|_{L^2[0,T]}^2=\|DY_t\|_{L^2[0,T]}^2={\rm Var}(Y_t). $$
For any measurable function $\phi $ bounded by $1,$ we apply Lemma \ref{dltq} to get
\begin{align*}
&|E\phi(\tilde{X}_{\varepsilon,t}) -E\phi(Y_t)|  \\
&\le C\left( E\|DY_t\|^{-8}_{L^2[0,T]} E\bigg( \int_0^t\int_0^t |D_{\theta}D_rY_t|^2d\theta dr\bigg)^2+ \left(E\|DY_t\|^{-2}_{L^2[0,T]}\right)^2\right)^{\frac{1}{4}}\|\tilde{X}_{\varepsilon,t}-Y_t\|_{1,2}\\
&\le \Bigg(\bigg(\sup_{0\leq u\leq t}\int_0^uk_2^2(u,s)ds\bigg)+\bigg(\sup_{0\leq u\leq t}\int_0^uk_3^2(u,s)ds\bigg)\Bigg)^{\frac{1}{2}}\left( \sup_{0\leq u\leq t}\int_0^{u}k_1^2(u,s)ds\right)^{\frac{1}{2}}\frac{C\varepsilon}{\sqrt{{\rm Var}(Y_t)}}.
\end{align*}
Taking the supremum over all $\phi$ yields
\begin{align*}
&d_{TV}(\tilde{X}_{\varepsilon,t},Y_t)\\
&\le  \Bigg(\bigg(\sup_{0\leq u\leq t}\int_0^uk_2^2(u,s)ds\bigg)+\bigg(\sup_{0\leq u\leq t}\int_0^uk_3^2(u,s)ds\bigg)\Bigg)^{\frac{1}{2}}\left( \sup_{0\leq u\leq t}\int_0^{u}k_1^2(u,s)ds\right)^{\frac{1}{2}}\frac{C\varepsilon}{\sqrt{{\rm Var}(Y_t)}},
\end{align*}
where $C$ is a positive constant not depending on $t$ and $\varepsilon.$ Note that, in the above computations, we have implicitly assumed that ${\rm Var}(Y_t)\neq 0,$ as otherwise the conclusion is obvious.

The proof of Theorem \ref{theorem1} is complete.
\end{proof}
\subsection{Proof of Theorem \ref{theorem2}}

%The second main result of this research shows that the convergence rate is optimal order by computing the limit of $\dfrac{E\varphi(\tilde{X}_{\varepsilon,t}) -E\varphi(Y_t)}{\varepsilon}$ when $\varepsilon$ tends to zero. We consider the processes $(Z_t)_{t\ge 0}$ has form
%\begin{align}\label{hfk4}
%Z_t=\int_0^tb'(t,s,x_s)Z_sds+\int_0^tb''(t,s,x_s)Y_s^2ds+2\int_0^t\sigma'(t,s,x_s)Y_sdB_s,\mbox{  }t\in [0,T].
%\end{align}

%\begin{proof}
We will carry out the proof in two steps.

\noindent\textbf{Step 1.} In this step, we show that for all $p\ge 2$ then
\begin{align}\label{ct6}
\lim\limits_{\varepsilon\to 0}E\left|\frac{\tilde{X}_{\varepsilon,t}-Y_t}{\varepsilon}-\frac{1}{2}Z_t\right|^{p}=0,\mbox{  }0\le t\le T.
\end{align}
For each $t\in[0,T],$ similar to (\ref{ct3}), we have
\begin{align*}
\tilde{X}_{\varepsilon,t}-Y_t&= \int_0^t b'(t,s,x_s)\big(\tilde{X}_{\varepsilon,s}-Y_s\big)ds+\frac{1}{2\varepsilon}\int_0^t b''(t,s,x_s+\xi _1(\Xs-x_s))(\Xs-x_s)^2ds\\
&+\int_0^t\sigma'(t,s,x_s)\left(\Xs-x_s\right)dB_s+\frac{1}{2}\int_0^t \sigma''(t,s,x_s+\xi _2(\Xs-x_s))(\Xs-x_s)^2dB_s,
\end{align*}
where $\xi_1,\xi_2$ are random variables lying between $0$ and $1.$ Recalling the definition (\ref{hfk4}) of $Z_t,$ we deduce
 \begin{align*}
&\frac{\tilde{X}_{\varepsilon,t}-Y_t}{\varepsilon}-\frac{1}{2}Z_t= \int_0^t b'(t,s,x_s)\bigg(\frac{\tilde{X}_{\varepsilon,s}-Y_s}{\varepsilon}-\frac{1}{2}Z_s\bigg)ds\\
&+ \frac{1}{2}\int_0^t b''\left(t,s,x_s+\xi _1(\Xs-x_s)\right)\big(\tilde{X}_{\varepsilon,s}^2-Y_s^2\big)ds\\
&+\frac{1}{2} \int_0^t \left(b''(t,s,x_s+\xi _1(\Xs-x_s))-b''(t,s,x_s)\right)Y_s^2ds\\
&+\int_0^t\sigma'(t,s,x_s)\big(\tilde{X}_{\varepsilon,s}-Y_s\big)dB_s+\frac{1}{2\varepsilon}\int_0^t \sigma''(t,s,x_s+\xi _2(\Xs-x_s))(\Xs-x_s)^2dB_s,
\end{align*}
and hence,
\begin{align*}
\left|\frac{\tilde{X}_{\varepsilon,t}-Y_t}{\varepsilon}-\frac{1}{2}Z_t\right|^p\le 5^{p-1}
&\Bigg( \left|\int_0^t b'(t,s,x_s)\bigg(\frac{\tilde{X}_{\varepsilon,s}-Y_s}{\varepsilon}-\frac{1}{2}Z_s\bigg)ds\right|^p\\
&+\left|\frac{1}{2}\int_0^t b''\left(t,s,x_s+\xi _1(\Xs-x_s)\right)\big(\tilde{X}_{\varepsilon,s}^2-Y_s^2\big)ds\right|^p\\
&+\left|\frac{1}{2} \int_0^t \left(b''(t,s,x_s+\xi _1(\Xs-x_s))-b''(t,s,x_s)\right)Y_s^2ds\right|^{p}\\
&+\left|\int_0^t\sigma'(t,s,x_s)\big(\tilde{X}_{\varepsilon,s}-Y_s\big)dB_s\right|^p\\
&+\left|\frac{1}{2\varepsilon}\int_0^t \sigma''(t,s,x_s+\xi _2(\Xs-x_s))(\Xs-x_s)^2dB_s\right|^p
\Bigg).
\end{align*}
By using Assumptions \ref{assum1}-\ref{assum2} and the H\"{o}lder and Burkholder-Davis-Gundy inequalities, we get
\begin{align*}
E\left|\frac{\tilde{X}_{\varepsilon,t}-Y_t}{\varepsilon}-\frac{1}{2}Z_t\right|^p
&\le 5^{p-1}\left(E\left(t\int_0^t k_2^2(t,s)\left|\frac{\tilde{X}_{\varepsilon,s}-Y_s}{\varepsilon}-\frac{1}{2}Z_s\right|^2ds\right)^{\frac{p}{2}}\right.\\
&+\frac{1}{2^p}E\left(t\int_0^t k_3^2(t,s)|\tilde{X}_{\varepsilon,s}^2-Y_s^2|^2ds\right)^{\frac{p}{2}}\\
&+\frac{1}{2^p}E\left(t\int_0^t \left|b''(t,s,x_s+\xi _1(\Xs-x_s))-b''(t,s,x_s)\right|^{2}Y_s^4ds\right)^{\frac{p}{2}}\\
&+E\left(\int_0^t k_2^2(t,s)|\tilde{X}_{\varepsilon,s}-Y_s|^2ds\right)^{\frac{p}{2}}\\
&+\left.\frac{1}{2^p\varepsilon^p}E\left(\int_0^t k_3^2(t,s)\left|\Xs-x_s\right|^4ds\right)^{\frac{p}{2}}\right),\mbox{  }0\le t\le T.
\end{align*}
We put
\begin{multline*}
H_{1,\varepsilon}(t):=\frac{1}{2^p}E\left(t\int_0^t k_3^2(t,s)|\tilde{X}_{\varepsilon,s}^2-Y_s^2|^2ds\right)^{\frac{p}{2}}+E\left(\int_0^t k_2^2(t,s)|\tilde{X}_{\varepsilon,s}-Y_s|^2ds\right)^{\frac{p}{2}}\\+\frac{1}{2^p\varepsilon^p}E\left(\int_0^t k_3^2(t,s)\left|\Xs-x_s\right|^4ds\right)^{\frac{p}{2}},
\end{multline*}
and $H_{2,\varepsilon}(t):=\frac{1}{2^p}E\left(t\int_0^t \left|b''(t,s,x_s+\xi _1(\Xs-x_s))-b''(t,s,x_s)\right|^{2}Y_s^4ds\right)^{\frac{p}{2}}.$ As a consequence, for every $p\geq \max\left\{\frac{2\beta}{\beta-1},\frac{2\gamma}{\gamma-1}\right\}>2,$ we obtain
\begin{align*}
E\left|\frac{\tilde{X}_{\varepsilon,t}-Y_t}{\varepsilon}-\frac{1}{2}Z_t\right|^p&\le C\left(\int_0^t k_2^{\frac{2p}{p-2}}(t,s)ds\right)^{\frac{p}{2}-1}\int_0^t E\left|\frac{\tilde{X}_{\varepsilon,s}-Y_s}{\varepsilon}-\frac{1}{2}Z_s\right|^{p}ds+H_{1,\varepsilon}(t)+H_{2,\varepsilon}(t)\\
&\leq C\int_0^t E\left|\frac{\tilde{X}_{\varepsilon,s}-Y_s}{\varepsilon}-\frac{1}{2}Z_s\right|^{p}ds+H_{1,\varepsilon}(t)+H_{2,\varepsilon}(t),
\end{align*}
where $C$ is a positive constant depending only on $p,T$ and $L.$ Then, an application of Gronwall's lemma gives us
\begin{equation}\label{domi} E\left|\frac{\tilde{X}_{\varepsilon,t}-Y_t}{\varepsilon}-\frac{1}{2}Z_t\right|^p\leq H_{1,\varepsilon}(t)+H_{2,\varepsilon}(t)+\int_0^t (H_{1,\varepsilon}(s)+H_{2,\varepsilon}(s))e^{C(t-s)}ds,\mbox{  }0\le t\le T.
\end{equation}
Furthermore, we have
\begin{multline*}
H_{1,\varepsilon}(t)\leq\frac{t^{\frac{p}{2}}}{2^p}\left(\int_0^t k_3^{\frac{2p}{p-2}}(t,s)ds\right)^{\frac{p}{2}-1}\int_0^t E|\tilde{X}_{\varepsilon,s}^2-Y_s^2|^pds+\left(\int_0^t k_2^{\frac{2p}{p-2}}(t,s)ds\right)^{\frac{p}{2}-1}\int_0^tE|\tilde{X}_{\varepsilon,s}-Y_s|^pds\\+\frac{1}{2^p\varepsilon^p}\left(\int_0^t k_3^{\frac{2p}{p-2}}(t,s)ds\right)^{\frac{p}{2}-1}\int_0^t E\left|\Xs-x_s\right|^{2p}ds,\mbox{  }0\le t\le T,
\end{multline*}
which, combined with (\ref{EX1}) and (\ref{EXY}), implies that $H_{1,\varepsilon}(t)\to 0$ as $\varepsilon\to 0.$ Note that $H_{2,\varepsilon}(t)$ is bounded uniformly in $\varepsilon\in (0,1).$ Indeed,
$$H_{2,\varepsilon}(t)\leq \frac{t^{\frac{p}{2}}}{2^p}E\left(\int_0^t 4k_3^2(t,s)Y_s^4ds\right)^{\frac{p}{2}}\leq t^{\frac{p}{2}}\left(\int_0^t k_3^{\frac{2p}{p-2}}(t,s)ds\right)^{\frac{p}{2}-1}\int_0^tE|Y_s|^{2p}ds<\infty.$$
So, by the dominated convergence theorem, we also have $H_{2,\varepsilon}(t)\to 0$ as $\varepsilon\to 0.$

Letting $\varepsilon\to 0,$ we obtain (\ref{ct6}) from (\ref{domi}). This finishes the proof of Step 1.

\noindent\textbf{Step 2.} In this step, we prove (\ref{ct}). For simplicity, we write $\langle.,. \rangle$ instead of $\langle.,. \rangle_{L^2[0,T]}$ and $\|.\|$ instead of $\|.\|_{L^2[0,T]}.$ Fixed $t\in (0,T].$ In view of the relation (3.2) in \cite{Dung2022}, we have
\begin{align*}
E[\varphi(\tilde{X}_{\varepsilon,t})]-E[\varphi(Y_t)]&=E\left[\int_{Y_t}^{\tilde{X}_{\varepsilon,t}} \varphi(z)dz\delta\bigg(\frac{DY_t}{\|DY_t\|^2}\bigg)\right]-E\left[\frac{\varphi(\tilde{X}_{\varepsilon,t}) \langle D\tilde{X}_{\varepsilon,t}-DY_t, DY_t \rangle}{\|DY_t\|^2}\right].
\end{align*}
We observe from the equation (\ref{DY}) that $D_{\theta}Y_t$ is deterministic for all $0\le \theta\le t\le T,$ and $E[D_{\theta}Y_t]=0.$ Hence, by Clark-Ocone formulla, we have
$$ Y_t=\int_0^tE[D_\theta Y_t|\mathcal{F}_\theta]dB_\theta=\int_0^t D_\theta Y_t dB_\theta,\,\,0\leq t\leq T.$$
The above relation implies that $\|DY_t\|^2={\rm Var}(Y_t)$ and $\delta\left(\frac{DY_t}{\|DY_t\|^2}\right)=\frac{Y_t}{{\rm Var}(Y_t)}.$ So we obtain
\begin{align*}
E[\varphi(\tilde{X}_{\varepsilon,t})]-E[\varphi(Y_t)]&=\frac{1}{{\rm Var}(Y_t)}E\left[Y_t\int_{Y_t}^{\tilde{X}_{\varepsilon,t}} \varphi(z)dz\right]-\frac{1}{{\rm Var}(Y_t)}E\left[\varphi(\tilde{X}_{\varepsilon,t}) \langle D\tilde{X}_{\varepsilon,t}-DY_t, DY_t \rangle\right].
\end{align*}
By the estimate (\ref{DXY}) and the limit (\ref{ct6}), we have $\frac{\tilde{X}_{\varepsilon,t}-Y_t}{\varepsilon}\to\frac{1}{2}Z_t$ in $L^2(\Omega)$ as $\varepsilon\to 0$
	 and $\max\limits_{\varepsilon>0}\frac{E\|D\tilde{X}_{\varepsilon,t}-DY_t\|^2}{\varepsilon^2}<\infty.$
	 Thus, it follows from Lemma $1.2.3$ in \cite{nualartm2} that $DZ_t$ exists for every $t\in [0,T]$ and
	 $\frac{D\tilde{X}_{\varepsilon,t}-DY_t}{\varepsilon}$ weakly converges to $\frac{DZ_t}{2}$ in $L^2(\Omega \times [0, T ])$ as $\varepsilon\to 0.$ Then, for every $\varepsilon\in(0,1),$ we have
\begin{align}
&\frac{E[\varphi(\tilde{X}_{\varepsilon,t})]-E[\varphi(Y_t)]}{\varepsilon}-\frac{1}{2{\rm Var}(Y_t)}E\left[\varphi(Y_t)Z_tY_t\right]+\frac{1}{2{\rm Var}(Y_t)}E\left[\varphi(Y_t)\langle DZ_t, DY_t \rangle\right]\notag\\
&=\frac{1}{{\rm Var}(Y_t)}E\left[\bigg(\frac{1}{\varepsilon}\int_{Y_t}^{\tilde{X}_{\varepsilon,t}} \varphi(z)dz-\frac{1}{2}\varphi(Y_t)Z_t\bigg)Y_t\right]\notag\\
&-\frac{1}{{\rm Var}(Y_t)}E\left[(\varphi(\tilde{X}_{\varepsilon,t}) -\varphi(Y_t))\bigg\langle \frac{D\tilde{X}_{\varepsilon,t}-DY_t}{\varepsilon}, DY_t\bigg\rangle\right]\notag\\
&-\frac{1}{{\rm Var}(Y_t)}E\left[\varphi(Y_t)\bigg\langle \frac{D\tilde{X}_{\varepsilon,t}-DY_t}{\varepsilon}-\frac{DZ_t}{2}, DY_t\bigg\rangle\right],\,\,0<t\leq T.\label{ct5}
\end{align}
For the first addend in the right hand side of (\ref{ct5}) we have
\begin{align*}
&\frac{1}{\varepsilon}\int_{Y_t}^{\tilde{X}_{\varepsilon,t}} \varphi(z)dz-\frac{1}{2}\varphi(Y_t)Z_t=\frac{\tilde{X}_{\varepsilon,t}-Y_t}{\varepsilon}\int_0^1 \varphi(Y_t+z(\tilde{X}_{\varepsilon,t}-Y_t))dz-\frac{1}{2}\varphi(Y_t)Z_t\\
&=\bigg(\frac{\tilde{X}_{\varepsilon,t}-Y_t}{\varepsilon}-\frac{Z_t}{2}\bigg)\int_0^1 \varphi(Y_t+z(\tilde{X}_{\varepsilon,t}-Y_t))dz+\frac{Z_t}{2}\int_0^1 \left(\varphi(Y_t+z(\tilde{X}_{\varepsilon,t}-Y_t))-\varphi(Y_t)\right)dz,
\end{align*}
and hence,
\begin{align*}
E\left|\bigg(\frac{1}{\varepsilon}\int_{Y_t}^{\tilde{X}_{\varepsilon,t}} \varphi(z)dz-\frac{1}{2}\varphi(Y_t)Z_t\bigg)Y_t\right|&\le \|\varphi\|_{\infty}E\left|\left(\frac{\tilde{X}_{\varepsilon,t}-Y_t}{\varepsilon}-\frac{Z_t}{2}\right)Y_t\right|\\
&+\frac{1}{2}E\left|Z_tY_t\int_0^1\left(\varphi(Y_t+z(\tilde{X}_{\varepsilon,t}-Y_t))-\varphi(Y_t)\right)dz\right|.
\end{align*}
Because the random variables $Y_t$ and $Z_t$ belong to $L^2(\Omega),$ recalling the limit (\ref{ct6}), we have
$$\lim\limits_{\varepsilon\to 0}E\left|\bigg(\frac{\tilde{X}_{\varepsilon,t}-Y_t}{\varepsilon}-\frac{Z_t}{2}\bigg)Y_t\right|=0.$$
By the dominated convergence theorem and the fact that $\tilde{X}_{\varepsilon,t}\to Y_t,$ we also have
$$\lim\limits_{\varepsilon\to 0}E\left|Z_tY_t\int_0^1 \left(\varphi(Y_t+z(\tilde{X}_{\varepsilon,t}-Y_t))-\varphi(Y_t)\right)dz\right|=0.$$
So it holds that
\begin{equation}\label{ct7}\lim\limits_{\varepsilon\to 0}E\bigg[\bigg(\frac{1}{\varepsilon}\int_{Y_t}^{\tilde{X}_{\varepsilon,t}} \varphi(z)dz-\frac{1}{2}\varphi(Y_t)Z_t\bigg)Y_t\bigg]=0.
\end{equation}
For the second addend in the right hand side of (\ref{ct5}) we have
\begin{align*}
E&\left[(\varphi(\tilde{X}_{\varepsilon,t}) -\varphi(Y_t))\bigg\langle \frac{D\tilde{X}_{\varepsilon,t}-DY_t}{\varepsilon}, DY_t\bigg\rangle\right]\leq \frac{1}{\sqrt{{\rm Var}(Y_t)}} E\left[\frac{|\varphi(\tilde{X}_{\varepsilon,t}) -\varphi(Y_t)|\| D\tilde{X}_{\varepsilon,t}-DY_t\|}{\varepsilon}\right]\\
&\leq \frac{1}{\sqrt{{\rm Var}(Y_t)}}(E|\varphi(\tilde{X}_{\varepsilon,t}) -\varphi(Y_t)|^2)^{\frac{1}{2}}\bigg(\frac{E\| D\tilde{X}_{\varepsilon,t}-DY_t\|^2}{\varepsilon^2}\bigg)^{\frac{1}{2}}.
\end{align*}
Once again, by the estimate (\ref{DXY}) and the dominated convergence theorem, we deduce
\begin{equation}\label{ct8}
\lim\limits_{\varepsilon\to 0}E\left[(\varphi(\tilde{X}_{\varepsilon,t}) -\varphi(Y_t))\bigg\langle \frac{D\tilde{X}_{\varepsilon,t}-DY_t}{\varepsilon}, DY_t\bigg\rangle\right]=0.
\end{equation}
For the last addend in the right hand side of (\ref{ct5}) we have
\begin{equation}\label{ct9}
\lim\limits_{\varepsilon\to 0}E\left[\varphi(Y_t)\bigg\langle \frac{D\tilde{X}_{\varepsilon,t}-DY_t}{\varepsilon}-\frac{DZ_t}{2}, DY_t\bigg\rangle\right]=0.
\end{equation}
because $\frac{D\tilde{X}_{\varepsilon,t}-DY_t}{\varepsilon}$ weakly converges to $\frac{DZ_t}{2}$ in $L^2(\Omega \times [0, T ])$ as $\varepsilon\to 0.$  Combining (\ref{ct7}), (\ref{ct8}) and (\ref{ct9}) yields
$$\lim\limits_{\varepsilon\to 0}\frac{E[\varphi(\tilde{X}_{\varepsilon,t})]-E[\varphi(Y_t)]}{\varepsilon}=\frac{1}{2{\rm Var}(Y_t)}E\left[\varphi(Y_t)Z_tY_t\right]
-\frac{1}{2{\rm Var}(Y_t)}E\left[\varphi(Y_t)\langle DZ_t, DY_t \rangle\right].$$
Moreover, by using the duality relationship (\ref{ct*}), we have
\begin{align*}
\delta (Z_tDY_t)=Z_tY_t-\langle DZ_t, DY_t \rangle
\end{align*}
Hence,
\begin{align*}
\lim\limits_{\varepsilon\to 0}\frac{E[\varphi(\tilde{X}_{\varepsilon,t})]-E[\varphi(Y_t)]}{\varepsilon}=\frac{1}{2{\rm Var}(Y_t)}E\left[\varphi(Y_t)\delta (Z_tDY_t)\right].
\end{align*}
The proof of Theorem \ref{theorem2} is complete.
%\end{proof}

\section{Stochastic Volterra equation with fractional Brownian motion kernel}\label{tglo}
Given $H\in (0,1),$ we consider the kernel
\[
	K_H(t,s):=\frac{(t-s)^{H-\frac{1}{2}}}{\Gamma(H-1/2)\sqrt{V_H}}F(H-1/2,1/2-H,H+1/2,1-t/s),\quad\text{$0\leq s< t\leq T$},	\]
where $ F$ is the Gauss hypergeometric function, $\Gamma$ denotes the usual Gamma function and $ V_H=\frac{\Gamma(2-2H)\cos\pi H}{\pi H(1-2H)}.$  We recall that the kernel $K_H$ can be used to define a fractional Brownian motion  as follows (see e.g. \cite{Coutin2001})
\begin{equation}\label{densityCIR02}
		B^H_t=\int_0^t K_H(t,s)d B_s,\,\,0\leq t\leq T.
\end{equation}
Moreover, the covariance function of $B^H$ is given by
	$$R_H(t,s):=E[B^H_t B^H_s]=\frac{1}{2}(t^{2H}+s^{2H}-|t-s|^{2H}),\,\,0\leq s\leq  t\leq T.$$
In this section, we consider the following stochastic Volterra equation with the fractional Brownian motion kernel $K_H$
\begin{equation}\label{pt11}
X_{\varepsilon,t}=x_0+\int_0^tK_H(t,s)b(s,X_{\varepsilon,s})ds+\varepsilon\int_0^tK_H(t,s)\sigma(s,X_{\varepsilon,s})dB_s,
\end{equation}
where $\varepsilon\in(0,1)$, the initial data $x_0\in\R$ and $b,\sigma: [0,T]\times\mathbb{R}\to \R$ are  measurable functions having linear growth, that is, there exists $M > 0$ such that
$$|b(t,x)|+|\sigma(t,x)|\le M(1+|x|)\,\,\forall x\in \R, 0\le t\le T.$$
Additionally, $b$ and $\sigma$ are twice differentiable functions in $x$ with the partial derivatives bounded by $M.$

The equation (\ref{pt11}) was studied first in \cite{Coutin2001}, the large and moderate deviation results have been discussed in \cite{Jacquier2022,Li2017,Zhang2008}. Here we apply the results of Section \ref{fkm} to bound the total variation distance for the fluctuation process of $X_{\varepsilon,t}.$
We define
\begin{align}
&x_t=x_0+\int_0^tK_H(t,s)b(s,x_s)ds,\mbox{  }0\le t\le T\notag.\\
&\tilde{X}_{\varepsilon,t}:=\frac{X_{\varepsilon,t}-x_t}{\varepsilon},\mbox{  }0\le t\le T,\notag
\end{align}
and
\begin{equation}\label{pt31}
Y_t=\int_0^tK_H(t,s)b'(s,x_s)Y_sds+\int_0^tK_H(t,s)\sigma(s,x_s)dB_s,\mbox{  }0\le t\le T.
\end{equation}

\begin{thm}For every $H\in (0,1),$ it holds that
\begin{equation}\label{aik}
d_{TV}(\tilde{X}_{\varepsilon,t},Y_t)\le \frac{Ct^{2H}\varepsilon}{\sqrt{{\rm Var}(Y_t)}}\,\,\mbox{  }\forall \varepsilon\in(0,1),\mbox{  }0< t\le T,
\end{equation}
where $C$ is a positive constant not depending on $t$ and $\varepsilon.$ When $H\geq \frac{1}{2}$ and $\sigma(t,x)\geq \sigma_0$ for some $\sigma_0>0$ and for all $(t,x)\in [0,T]\times \mathbb{R},$ we also have
\begin{equation}\label{aikh}
d_{TV}(\tilde{X}_{\varepsilon,t},Y_t)\le Ct^{H}\varepsilon\,\,\mbox{  }\forall \varepsilon\in(0,1),\mbox{  }0< t\le T.
\end{equation}
\end{thm}
\begin{proof} We recall from \cite{Decreusefond1999} that $0\leq K_H(t,s)\leq c(t-s)^{H-\frac{1}{2}}s^{-|H-\frac{1}{2}|}$ for some $c>0.$ Hence, it is easy to verify that the coefficients of the equation (\ref{pt11}) satisfy Assumptions  \ref{assum1} and \ref{assum2} with $k_1(t,s)=k_2(t,s)=k_3(t,s)=MK_H(t,s)$ and
\begin{align*}
\sup_{0\le u\le t}\int_0^uk^2_i(u,s)ds&= M^2\sup_{0\le u\le t}\int_0^uK_H^2(u,s)ds \\
&=M^2\sup_{0\le u\le t}E|B^H_u|^2=M^{2}t^{2H},\,i=1,2,3.
\end{align*}
So the estimate (\ref{aik}) follows directly from Theorem \ref{theorem1}.

We now consider the case $H\geq \frac{1}{2}.$ In this case, the kernel $K_H$ reduces to the following (see e.g. \cite{nualartm2} pp. 277-279)
$$
K_H(t,s):=c_H\,s^{1/2-H}\int_s^t (u-s)^{H-\frac{3}{2}}u^{H-1/2}du,\quad\text{$0\leq s< t\leq T,$}
$$
where $
c_H=\sqrt{\frac{H(2H-1)}{\beta(2-2H,H-1/2)}}$  with $\beta$ is the Beta function. We have the following estimates
\begin{align}
\int_s^{t}K^{2}_H(t,\theta)d\theta&=c_H^2\int_s^t \theta^{1-2H}\left(\int_\theta^t(u-\theta)^{H-3/2}u^{H-1/2}du\right)^2d\theta\notag\\
&\ge c_H^2\int_s^t \left(\int_\theta^t(u-\theta)^{H-3/2}du\right)^2d\theta=\frac{2c_H^2}{H(2H-1)^2}(t-s)^{2H}\label{41}
\end{align}
and
\begin{align}
|t-s|^{2H}&=E|B_t^H-B_s^H|=E\left(\int_0^{s}\big(K_H(t,\theta)-K_H(s,\theta)\big)dB_\theta+\int_s^tK_H(t,\theta)dB_\theta\right)^2\notag\\
&=E\left(\int_0^{s}\big(K_H(t,\theta)-K_H(s,\theta)\big)dB_\theta\right)^2+E\left(\int_s^tK_H(t,\theta)dB_\theta\right)^2\notag\\
&\ge E\left(\int_s^tK_H(t,\theta)dB_\theta\right)^2=\int_s^tK_H^2(t,\theta)d\theta,\, 0\le s\le t\le T.\label{42}
\end{align}
On the other hand, from the equation (\ref{pt31}) we have
\begin{align*}
D_\theta Y_t=K_H(t,\theta)\sigma(\theta,x_\theta)+\int_\theta^tK_H(t,s)b'(s,x_s)D_\theta Y_sds,\,\,0\leq \theta\leq t\leq T.
\end{align*}
%%%%%%%%%%%%%%%%%%%%%%%%%%%%%%%%%%%%%%%%%%%%%%%%%%%%5
Hence, for some $C>0,$
\begin{align*}
|D_\theta Y_t|^2&\le 2M^2(1+|x_\theta|)^2K_H^2(t,\theta)+2M^2\int_\theta^tK_H^2(t,s)ds\int_\theta^t|D_\theta Y_s|^2ds\\
&\le C K_H^2(t,\theta)+2M^2|t-\theta|^{2H}\int_\theta^t|D_\theta Y_s|^2ds\\
&\le C K_H^2(t,\theta)+2M^2T^{2H}\int_\theta^t|D_\theta Y_s|^2ds\\
&\le C K_H^2(t,\theta)+C\int_\theta^t|D_\theta Y_s|^2ds.
\end{align*}
By Gronwall's lemma and the fact that the function $s\to K_H(s,\theta)$ is non-decreasing, we have
\begin{align*}
|D_\theta Y_t|^2&\le C K_H^2(t,\theta)+C\int_\theta^te^{C(s-\theta)}K^2_H(s,\theta)ds\\
&\le C K_H^2(t,\theta)+Ce^{CT}\int_\theta^tK^2_H(s,\theta)ds\\
&\le C K_H^2(t,\theta),\,0\le \theta\le t\le T,
\end{align*}
where $C$ is a positive constant not depending on $t.$ For $h\in (0,1],$ by the fundamental inequality $(a+b)^2\ge \frac{1}{2}a^2-b^2$, we have
\begin{align*}
{\rm Var}(Y_t)&=\int_0^{t}\left|D_\theta Y_t\right|^2d\theta \ge\int_{t(1-h)} ^{t}\left|D_\theta Y_t\right|^2d\theta\\
&\geq \frac{1}{2}\int_{t(1-h)} ^{t}K^2_H(t,\theta)\sigma^2(\theta,x_\theta)d\theta-\int_{t(1-h)} ^{t}\bigg(\int_\theta^tK_H(t,s)b'(s,x_s)D_\theta Y_sds\bigg)^2d\theta\\
&\ge \frac{\sigma_0^2}{2} \int_{t(1-h)} ^{t}K_H^2(t,\theta)d\theta-C\int_{t(1-h)} ^{t}\bigg(\int_\theta^tK_H^{2}(t,s)ds \int_\theta^tK_H^2(s,\theta)ds\bigg) d\theta.
\end{align*}
%For $H>\frac{1}{2},$
From (\ref{41}) and (\ref{42}), we deduce
\begin{align*}
{\rm Var}(Y_t)&\ge \frac{\sigma_0^2c_H^2}{H(2H-1)^2}(th)^{2H}-C\int_{t(1-h)} ^{t}|t-\theta|^{2H}\int_\theta^tK_H^2(s,\theta)dsd\theta\\
&\ge \frac{\sigma_0^2c_H^2}{H(2H-1)^2}(th)^{2H}-C(th)^{2H}\int_{t(1-h)} ^{t}\int_{t(1-h)}^s K_H^2(s,\theta)d\theta ds\\
&\ge  \frac{\sigma_0^2c_H^2}{H(2H-1)^2}(th)^{2H}-C(th)^{4H+1}=(th)^{2H}(c_H^{*}-C(Th)^{2H+1}),
\end{align*}
where $c_H^{*}=\frac{\sigma_0^2c_H^2}{H(2H-1)^2}.$ By choosing $h=1\wedge \sqrt[2H+1]{\frac{c_H^{*}}{2CT^{2H+1}}},$ we get
$$ {\rm Var}(Y_t)\ge \frac{c_H^{*}t^{2H}}{2},\,\,0\leq t\leq T. $$
Inserting the above relation into (\ref{aik}) gives us (\ref{aikh}). The proof of the theorem is complete.
\end{proof}

%\noindent {\bf  CRediT authorship contribution statement}  N.T. Dung: Writing--review \& editing, Writing--original draft. N.T. Hang: Writing--review \& editing, Writing--original draft.

%\noindent {\bf  Declaration of competing interest} The authors declare that they have no known competing financial interests or personal relationships that could have appeared to influence the work reported in this paper.

%\noindent {\bf Data Availability} No datasets were generated or analysed during the current study.

%

%\noindent {\bf Data Availability} No datasets were generated or analysed during the current study.

%\noindent {\bf \large Declarations}

%\noindent {\bf Competing Interests} The authors declare no competing interests.

%\noindent {\bf Acknowledgments.} This research is funded by Ha Noi University of Mining and Geology, Vietnam, under Grant No. T25-17.

%\noindent {\bf Disclosure statement} No potential conflict of interest was reported by the authors.

\end{document}